\documentclass[12pt]{amsart}
\usepackage[english]{babel}
\usepackage{tikz-cd}
\usepackage{cite}
\usepackage{url}
\usepackage{hyperref}
\usepackage{anysize}
\usepackage{graphicx}
\usepackage{fullpage}
\usepackage{colonequals}
\usepackage{amssymb}
\usepackage[leqno]{amsmath}
\usepackage{amscd}
\usepackage{amsthm}
\usepackage{amsmath}
\usepackage{enumitem}
\usepackage{xcolor}
\usepackage{marvosym}
\usepackage{soul} %para tachar
\usepackage{amsaddr}
\usepackage{amsthm} % lo agregué para cmabiar enumeración de teoremas
\usepackage{mathrsfs} %mas letras caligráficas

\marginsize{3cm}{3cm}{1cm}{2cm}
\def\R{\mathbb{R}}
\def\Q{\mathbb{Q}}
\def\A{\mathbb{A}}
\def\Z{\mathbb{Z}}
\def\C{\mathbb{C}}
\def\O{\mathcal{O}}
\def\H{\mathcal{H}}

\def\TT{\boldsymbol{T}}
\def\GG{\boldsymbol{G}}
\def\PGL{\mathrm{PGL}}
\def\PSL{\mathrm{PSL}}
\def\PSO{\mathrm{PSO}}

\def\KK{\boldsymbol{K}}

\def\p{\mathfrak{p}}

\def\A{\mathbb{A}}
\def\P{\mathbb{P}}
\def\KK{\boldsymbol{K}}

\usepackage{amsmath}
\usepackage[all,cmtip]{xy}
\newtheorem{theorem}{Theorem}[section]
\newtheorem*{theorem*}{Lema}

\newtheorem{definition}{Definition}
\newtheorem{lemma}[theorem]{Lemma}
\newtheorem{proposition}[theorem]{Proposition}

\theoremstyle{remark}
\newtheorem{remark}{Remark}

\numberwithin{equation}{section}

\newenvironment{customtheorem}[1]{%
  \renewcommand\thetheorem{#1}
  \theorem
}{\endtheorem}

\newcommand\rquot[2]{
  \mathchoice
  {% \displaystyle
    \text{\raise0.5ex\hbox{$#1$}\big/\lower0.5ex\hbox{$#2$}}%
  }
  {% \textstyle
    #1\,/\,#2
  }
  {% \scriptstyle
    #1\,/\,#2
  }
  {% \scriptscriptstyle
    #1\,/\,#2
  }
}

\newcommand\lrquot[3]{
  \mathchoice
  {% \displaystyle
    \text{\lower0.5ex\hbox{$#1$}\big\backslash\raise0.5ex\hbox{$#2$\!}\big/
      \lower0.5ex\hbox{\!\!$#3$}}%
  }
  {% \textstyle
    #1\,\backslash\,#2\,/\,#3
  }
  {% \scriptstyle
    #1\,\backslash\,#2\,/\,#3
  }
  {% \scriptscriptstyle
    #1\,\backslash\,#2\,/\,#3
  }
}

\newcommand\lquot[2]{
  \mathchoice
  {% \displaystyle
    \text{\lower0.5ex\hbox{$#1$}\big\backslash\raise0.5ex\hbox{$#2$}}%
  }
  {% \textstyle
    #1\,\backslash\,#2
  }
  {% \scriptstyle
    #1\,\backslash\,#2
  }
  {% \scriptscriptstyle
    #1\,\backslash\,#2
  }
}

\DeclareFontFamily{U}{wncy}{}
\DeclareFontShape{U}{wncy}{m}{n}{<->wncyr10}{}
\DeclareSymbolFont{mcy}{U}{wncy}{m}{n}
\DeclareMathSymbol{\Sh}{\mathord}{mcy}{"58}
\DeclareFontFamily{U} {cmmi}{}

\DeclareFontShape{U}{cmmi}{m}{n}{
  <-6> cmmi5
  <6-7> cmmi6
  <7-8> cmmi7
  <8-9> cmmi8
  <9-10> cmmi9
  <10-12> cmmi10
  <12-> cmmi12}{}

\DeclareSymbolFont{Xcmmi} {U} {cmmi}{m}{n}

\DeclareMathSymbol{\xi}{\mathord}{Xcmmi}{24}

\title{Equidistribution of Stark-Heegner and ATR and cycles}
\author{Patricio Pérez-Piña}
\address{Pontificia Universidad Católica de Chile}
\email{perezpinha.patricio@gmail.com}
\date{\today}

\begin{document}
\begin{abstract}
We prove the equidistribution of some cycles of $S$-arithmetic nature that are related to RM points and Stark-Heegner points. We also prove the equidistribution of Picard orbits of ATR cycles as defined by Darmon, Rotger and Zhao.
\end{abstract}

\maketitle

\section{Introduction}

We denote by $Y$ the modular curve of level $1$ over $\Q$. Let $\H$ be the complex upper half plane and recall that $\PSL_2(\Z)$ acts on $\H$ by M\"obius transformations. The quotient space is identified with the complex points of the modular curve, in other words \begin{equation}\label{modularcurve}\PSL_2(\Z)\backslash\H\cong Y(\C).\end{equation}

Let $K/\Q$ be a quadratic extension and denote the discriminant (resp. class number) of $K$ by $d_K$ (resp. $h_K$). The ring of integers of $K$ is denoted by $\O_K$. Consider $\psi\colon K\to M_2(\Q)$ an algebra embedding such that $\psi(K)\cap M_2(\Z)=\psi(\O_K)$. 

When $K$ is imaginary, $\psi(K^\times)$ has a unique fixed point in $\H$ that we denote by $z_\psi$. Define $\Omega_K$ as the collection of $\PSL_2(\Z)$-orbits in \[\left\{z_\psi\in\H\mid \psi\colon K\to M_2(\Q) \mbox{ and }\psi(K)\cap M_2(\Z)=\psi(\O_K)\right\}.\] Under \eqref{modularcurve}, $\Omega_K$ is identified with the $h_K$ elliptic curves in $Y(\C)$ with CM by $\O_K$. Moreover, from the theory of complex multiplication, it is known that $\Omega_K$ is a Galois orbit.

If $K$ is real, $\psi(K^\times)$ has two fixed points in $\R$ and they determine the endpoints of a unique geodesic in $\H$ that we denote by $\mathcal{Y}_{\psi}$. In this case, $\Omega_K$ will denote the collection of $\PSL_2(\Z)$-orbits in \[\{\mathcal{Y}_{\psi}\subseteq\H\mid \psi\colon K\to M_2(\Q) \mbox{ and }\psi(K)\cap M_2(\Z)=\psi(\O_K)\}.\] In this case, under \eqref{modularcurve}, $\Omega_K$ is identified with a collection of $h_K$ closed geodesics in $Y(\C)$. Moreover, $\Omega_K$ is an orbit of the class group of $K$ (see \cite{ELMV2}).

Let $\mu_{hyp}$ be the hyperbolic probability measure on $\PSL_2(\Z)\backslash\H$ and $ds=\mathrm{Im}(z)^{-1}dz$ the hyperbolic metric. 

\begin{theorem}[Theorem 1 in \cite{D}] As $|d_K|\to\infty$, the collection $\Omega_K$ becomes equidistruted on $\PSL_2(\Z)\backslash\H$. In other words, for every compactly supported continuous function $f\colon \PSL_2(\Z)\backslash\H\to\C$:
\begin{enumerate}
	\item In the imaginary case, \[\frac{1}{h_K}\sum_{z_\psi\in \Omega_K}f(z_\psi)\to\int fd\mu_{hyp} \mbox{ as } d_K\to-\infty.\]
	\item In the real case, \[\frac{\sum_{\mathcal{Y_\psi}\in \Omega_K}\int_{\mathcal{Y}_\psi}f(s)ds}{\sum_{\mathcal{Y}_\psi\in \Omega_K}\mathrm{length}(\mathcal{Y}_\psi)}\to\int fd\mu_{hyp} \mbox{ as } d_K\to\infty.\]
\end{enumerate}
\end{theorem}

In this work we establish two generalizations of Duke's Theorem that we proceed to explain. Let $p$ be a prime number and let $\tau$ be a real quadratic irrationality over $\Q$. We say that $\tau$ is a \emph{real multiplication} (RM) point if $\tau$ generates a field inert or ramified at $p$. Consider $E/\Q$ an elliptic curve with multiplicative reduction at $p$. In \cite{Darmonintegration}, the author provides a conjectural construction of algebraic points on $E$ using RM points. The process of doing this relies on the concept of $p$-adic integration with respect to a so-called mock Hilbert modular form. The points arising from this construction are known as Stark-Heegner points and they are indexed by cycles that we call Stark-Heegner cycles. Such collection of points are expected to be Galois orbits, see Conjectures 5.3 and 5.6 in \cite{Darmonintegration}. On the other hand, in \cite{Darmonvonk} the authors propose a construction, relying on the notion of rigid meromorphic cocycles for the Ihara group $\PSL_2(\Z[1/p])$ and using RM points, which is expected to provide an explicit class field theory for real quadratic fields, see Conjecture 1 in \cite{Darmonvonk}.

Now we describe what we call Stark-Heegner cycles. Let $\mathscr{H}$ be the Drinfeld's upper half plane as in \cite{DasTei}. It is the rigid analytic space whose $L$-points on a complete field extension $L/\Q_p$ are given by $\mathscr{H}(L)=\P^1(L)\smallsetminus\P^1(\Q_p)$. Let $\Q_{p^2}$ denote the unique quadratic unramified extension of $\Q_p$. We set the notations $\H_p=\mathscr{H}(\C_p)$ and $\H_{p^2}^{unr}=\mathscr{H}(\Q_{p^2})$. The diagonal action of the Ihara group $\PSL_2(\Z[1/p])$ on $\H\times\H_p$ by M\"{o}bius transformations is discrete. The quotient space \[\lquot{\PSL_2(\Z[1/p])}{(\H\times\H_p)}\] may be referred to as a mock Hilbert modular surface (see \cite{darmonmock}). This space can naturally be split in two open sets. Indeed, there exists a $\PGL_2(\Q_p)$-equivariant map from $\H_p$ to the set of vertices and edges of the Bruhat-Tits tree of $\PGL_2(\Q_p)$ known as the reduction map. Using this, one can write $\H_p$ as a disjoint union of two $\PSL_2(\Q_p)$-invariants sets $\H_p^{even}$ and  $\H_p^{odd}$, where $\H_p^{even}$ (resp. $\H_p^{odd}$) denote the set of points in $\H_p$ that are reduced to the set of even (resp. odd) vertices or edges. Therefore, the mock Hilbert modular surface equals \[\lquot{\PSL_2(\Z[1/p])}{(\H\times\H_p^{even})}\sqcup\lquot{\PSL_2(\Z[1/p])}{(\H\times\H_p^{odd})}.\] The space $\lquot{\PSL_2(\Z[1/p])}{(\H\times\H_p^{even})}$ can also be expressed as the quotient space $\lquot{\PGL^+_2(\Z[1/p])}{(\H\times\H_p)}$. In this text, we will prefer the latter description.

Let $K/\Q$ be a real quadratic extension inert at $p$. Consider an algebra embedding $\psi\colon K\to M_2(\Q)$. Then, there is an action of $K^\times$ on both $\H$ and $\H_p$ induced by the map $\psi\colon K^\times\to \mathrm{GL}_2(\Q)$. Since $K$ is real, $K^\times$ fixes two conjugate real quadratic irrationalities on the boundary $\P^1(\R)$ of $\H$. We denote by $\tau_\psi$ the fixed point for which $\psi(x)$ acts on the column $(\tau\,\,\,\, 1)^t$ as multiplication by $x$ for every $x$ in $K$. Note that this element is actually in $K$ so it defines an RM point. On the other hand, since $K$ is inert at $p$, the group $K^\times$ fixes two points in $\H_p$ that in fact must also be $\tau_\psi$ and its Galois conjugate $\tau'_\psi$. Let $\mathcal{Y}_\psi$ be the unique geodesic in $\H$ connecting $\tau_\psi$ and $\tau'_\psi$.

\begin{definition} The Stark-Heegner cycle attached to $\psi$ is the image of $\mathcal{Y}_\psi\times\{\tau_\psi\}$ in $\PGL^+_2(\Z[1/p])\backslash (\H\times\H_p)$. We denote it by $\Delta_\psi$.
\end{definition}

Let $\O$ be an order in $K$. We say that the Stark-Heegner cycle $\Delta_\psi$ has conductor $\O[1/p]$ if $\psi(K)\cap M_2(\Z[1/p])=\psi(\O[1/p])$. The narrow class group $Pic^+(\O[1/p])$ parametrizes the Stark-Heegner cycles of conductor $\O[1/p]$. We denote by $disc_p(\O)$ the prime-to-$p$ part of the discriminant of $\O$.

Observe that $\tau_\psi$ lies on $K\cap \H_p\subseteq \H_{p^2}^{unr}$ and so the cycle $\Delta_\psi$ is confined inside the closed subspace $\PGL^+_2(\Z[1/p])\backslash(\H\times\H_{p^2}^{unr})$ of the mock Hilbert modular surface. Each $\Delta_\psi$ is compact and comes equipped with a unique probability measure $\nu_\psi$ invariant under the geodesic flow on its $\H$-coordinate. Denote by $\mu_{\O[1/p]}$ the unique probability measure on $\PGL^+_2(\Z[1/p])\backslash(\H\times\H_{p^2}^{unr})$ proportional to $\sum\nu_\psi$, where the sum runs trough those $\psi$ for which $\Delta_\psi$ has conductor $\O[1/p]$. Denote by $\mu$ the unique probability measure on $\PGL^+_2(\Z[1/p])\backslash(\H\times\H_{p^2}^{unr})$ coming from a Haar measure on $\PGL^+_2(\R)\times\PGL_2(\Q_p)$ and a counting measure on $\PGL^+_2(\Z[1/p])$.

\begin{customtheorem}{A}[]\label{equidistributionSH} The Stark-Heegner cycles of conductor $\O[1/p]$, when ordered by $disc_p(\O)$, become equidistributed on the space $\PGL^+_2(\Z[1/p])\backslash(\H\times\H_{p^2}^{unr})$ with respect to $\mu$. In other words, for every $f\colon \PGL^+_2(\Z[1/p])\backslash(\H\times\H_{p^2}^{unr})\to\C $ continuos and compactly supported function, we have that \[\lim_{disc_p(\O)\to\infty}\int_{\PGL^+_2(\Z[1/p])\backslash(\H\times\H_{p^2}^{unr})}fd\mu_{\O[1/p]}=\int_{\PGL^+_2(\Z[1/p])\backslash(\H\times\H_{p^2}^{unr})}fd\mu.\]
\end{customtheorem}

The natural projection of the cycles $\Delta_\psi$ to $\PSL_2(\Z[1/p])\backslash\H_{p}$ is the Picard orbit of an RM point $\tau_\psi$ as studied in \cite{Darmonvonk}. Therefore, one can conceive Theorem \ref{equidistributionSH} as a distribution statement about RM points. Note that a traditional equidistribution statement of RM points on $\PSL_2(\Z[1/p])\backslash\H_{p}$ would not make sense since $\PSL_2(\Z[1/p])$ acts with dense orbits on $\H_{p^2}^{unr}$. A similar issue appears when considering a $p$-adic analogue of Duke's theorem on the equidistribution of geodesics (see \cite{paper1}).

Now we proceed to define the ATR cycles introduced by Darmon, Roter and Zhao and state our main result about them. Let $K/F$ be a quadratic extension. Write $[F\colon\Q]=r+1$ with $r\geq0$ and enumerate the embeddings $F\to\R$ as $\sigma_0,...,\sigma_{r}$. For $0\leq j\leq r$, denote by $K_{\sigma_j}$ the algebra $K\otimes_F\R$ with $\R$ viewed as an $F$-algebra through the embedding $\sigma_j$.

\begin{definition} The extension $K/F$ is called an almost totally real (ATR) extension if
\begin{enumerate}[font=\normalfont] 
\item $K_{\sigma_0}\cong\C$.
\item $K_{\sigma_j}\cong\R\times\R$ for $1\leq j\leq r$.
\end{enumerate}
\end{definition}

Let $E$ be an elliptic curve defined over a totally real field $F$ of narrow class number $h_F^+=1$. In chapter 8 of \cite{Dbook}, the author describes a process of integration of the Hilbert modular form attached to $E$ using ATR extensions. Conjecturally, this leads to the construction of algebraic points on $E$. In the special case when $F=\Q$ this recovers the classical construction of Heegner points on $E$ (e.g., as in \cite{GZ} section V.2). When $E$ has everywhere good reduction and $F$ is a quadratic field, \cite{DL} provides numerical evidence for these conjectures. The construction involves computing the image under a complex Abel-Jacobi map of some non-algebraic cycles, called ATR cycles, lying inside the Hilbert modular variety attached to $F$. 

Denote by $\Z_F$ the ring of integers of $F$. For $0\leq j\leq r$, define $\mathcal{H}_j=\mathcal{H}$ the upper half plane and $\mathcal{H}^{r+1}=\prod_{j}\mathcal{H}_j$, which is endowed with an action of $\PGL_2(\R)^{n+1}$ given coordinate-wise. The natural embedding $(\sigma_0,...,\sigma_r)\colon M_2(F)\to M_2(\R)^{r+1}$ induces an embedding of $\PGL_2(F)$ inside $\PGL_2(\R)^{r+1}$, thus inducing an action of $\PSL_2(\Z_F)$ on $\H^{r+1}$. The action is properly discontinuous and the quotient $\lquot{\PSL_2(\Z_F)}{\H^{r+1}}$ is the space of complex points of a Hilbert modular variety that parametrizes abelian varieties of dimension $r+1$ with real multiplication by $\Z_F$.

Let $\psi\colon K\to M_2(F)$ be an $F$-algebra embedding from an ATR extension $K/F$. Use $\sigma_0$ to view $K\subseteq \C$. Then $\sigma_j\circ \psi$ induces an action of $K^\times$ on $\H_j$.  Condition (1) implies that $K^\times$ fixes a unique point $\tau_0\in\H_0$ and (2) implies for $1\leq j\leq r$ the existence of a unique geodesic $\mathcal{Y}_j\subseteq\H_j$ preserved by $K^\times$. The geodesic connects the two fixed points of $K^\times$ in the boundary $\mathbb{P}^1(\R)$ of $\mathcal{H}_j$ ($1\leq j\leq r$). Observe that since $\psi(K^\times)$ is defined over $F$, we have that $\tau_0\in K$ and the fixed points at $\partial\H_j$ correspond to $\sigma_j(\tau_0)$ and its Galois conjugate. 

\begin{definition}[See \cite{DRZ}] The ATR cycle attached to $\psi$ is the image of the set $\{\tau_0\}\times\prod_{j=1}^r\mathcal{Y}_j$ in $\lquot{\PSL_2(\Z_F)}{\H^{r+1}}$ . We denote it by $\Delta_\psi$.
\end{definition}

Let $\O$ be an order in $K$. We say that the ATR cycle $\Delta_\psi$ has conductor $\O$ if $\psi(K)\cap M_2(\Z_F)=\psi(\O)$. Let $Pic^+(\O)$ be the narrow class group of $\O$. The number of ATR cycles of conductor $\O$ is exactly $h^+(\O)=|Pic^+(\O)|$ (see \ref{ATRcyclespar}). It is expected that the collection of $\Delta_\psi$ with a given conductor gives rise to a Galois orbit in $E$, see Conjecture 2.2 un \cite{DRZ}. 

Denote by $A$ the diagonal group of $\PGL_2$. Each cycle $\Delta_\psi$ is compact and comes equipped with a unique $A(\R)^r$-right invariant probability measure $\nu_\psi$. We denote by $\mu_\O$ the unique probability measure proportional to $\sum_{}\nu_{\psi}$. We see $\mu_\O$ as a probability measure on $\PSL_2(\Z_F)\backslash \H^{r+1}$ supported over the collection of ATR cycles of conductor $\O$. Denote by $\mu$ the unique probability measure on $\PSL_2(\Z_F)\backslash \H^{r+1}$ coming from a Haar measure on $\PGL^+_2(\R)^{r+1}$ and a counting measure in $\PSL_2(\Z_F)$. Denote by $\mathfrak{d}_{\O/\Z_F}$ the discriminant $\Z_F$-ideal of $\O$.

\begin{customtheorem}{B}[]\label{equidistributionATR} The ATR cycles of conductor $\O$, when ordered by the absolute value of $N_{F/\Q}(\mathfrak{d}_{\O/\Z_F})$, become equidistributed on the Hilbert modular variety $\PSL_2(\Z_F)\backslash \H^{r+1}$. In other words, for every continuos and compactly supported function \break $f\colon \PSL_2(\Z_F)\backslash \H^{r+1}\to\C $, we have \[\lim_{|N_{F/\Q}(\mathfrak{d}_{\O/\Z_F})|\to\infty}\int_{\PSL_2(\Z_F)\backslash \H^{r+1}}fd\mu_{(\O)}=\int_{\PSL_2(\Z_F)\backslash \H^{r+1}}fd\mu.\]
\end{customtheorem}

It is worth mentioning that the assumption $h_F^+=1$ is not essential but it makes the exposition clearer. The construction of ATR cycles in a more general case can be found in \cite{gartner}.

\subsection{Outline} The main tool for our purposes is Theorem 4.6 in \cite{ELMV3}. In consequence, our strategy is to describe the previous objects and ambient spaces in an $S$-arithmetic and adelic context. In section \ref{toralsets} we review the language and tools necessary to prove Theorem \ref{equidistributionSH} and \ref{equidistributionATR}. We start by recalling the definition of a homogeneous toral set and its discriminant. Later, we state a general result on the equidistribution of homogeneous toral sets in quotients of adelic homogeneous spaces (Theorem \ref{equidistribution}). The main reference for this is \cite{ELMV3}. To parametrize both Stark-Heegner and ATR cycles, we will use optimal embeddings. These are introduced in section \ref{embeddings}. In section \ref{ATR} we focus on its application to ATR cycles. We explain the parametrization by the group $Pic^+(\O)$ and how the ATR cycles are realized as projection of homogeneous toral sets. In this fashion, Theorem \ref{equidistributionATR} is obtained as a consequence of the previous sections. Lastly, in a similar way, section \ref{StarkH} shows how to apply section \ref{equidistribution} to prove Theorem \ref{equidistributionSH}.

\section*{Acknowledgments}

I am very grateful to my advisor Ricardo Menares for his support and careful reading of this manuscript. I also would like to thank Manuel Luethi for very useful conversations and references. Finally, I thank Philippe Michel and the Analytic Number Theory group at EPFL. I was very lucky to enjoy their hospitality during the academic period 2022-2023. This work was supported by ANID Doctorado Nacional No 21200911 and by the Bourse d’excellence de la Confédération suisse No. 2022.0414.

\section{Equidistribution of Homogeneous toral sets}\label{toralsets}

\subsection*{Notation} Let $\p$ be a finite prime of $F$. We use $F_\p$ and $\Z_{F,\p}$ to denote the completion with respect to the $\p$-adic topology of $F$ and $\Z_F$ respectively. If $R$ is an $\Z_F$-algebra, $R_\p$ denotes $R\otimes \Z_{F,\p}$. Let $S$ be a finite set of primes containing the infinite ones. We use $F_S$ to denote the algebra $\prod_{\p\in S}F_\p$ (including the infinite places). We use $\A_F$ for the ring of adeles over $F$. The algebra $\A^S$ is defined by the decomposition $\A_F=\A_F^S\times F_S$. When $S_\infty$ is the set of infinite places, we use the notation $\A_{F,f}$ to denote $\A_F^{S_\infty}$. Write $\widehat{\Z_F}$ to denote the closure of $\Z_F$ in $\A_{F,f}$. We use $R_S$, $R_\infty$ and $\widehat{R}$ to denote $R\otimes F_S$, $R_{S_\infty} $ and $R\otimes\widehat{\Z_F}$ respectively. The symbol $R_+$ stands for the totally positive elements in $R$. When dealing with subgroups of the general linear group, the superscript $+$ denotes elements of determinant totally positive.
\subsection{Orders}\label{orders}

Let $F$ be a totally real field with narrow class number $h_F^+=1$. In particular $F$ has trivial class number and $\Z_F$ is a principal ideal domain. Let $K/F$ be a quadratic extension. Then $\O_K$ is a free $\Z_F$-module of rank $2$ and the quotient $\O_K/\Z_F$ is torsion-free. This shows that $\O_K$ admit an $\Z_F$-basis containing $1$ and therefore $\O_K$ is monogenic. Fix $\omega_K\in\O_K$ such that $\O_K=\Z_F[\omega_K]$.

A $\Z_F$-order in $\O_K$ is completely determined by its index in $\O_K$. Indeed if $\O$ is such an order and $[\O_K\colon \O]=f\geq1$, then $\Z_F+f\O_K=\Z_F[f\omega_K]\subseteq \O$ and $\Z_F[f\omega_K]$ has also index $f$ so we have equality $\O=\Z_F[f\omega_K]$. We refer to $f$ as the conductor of $\O$ and we use the notation $\O_f\colonequals \Z_F[f\omega_K]$, the unique $\Z_F$-order of conductor $f$ in $\O_K$. The discriminant ideal $\mathfrak{d}_{\O_f/\Z_F}$ is the principal $\Z_F$-ideal generated by $f^2(\omega_K-\omega_K')^2$ with $\omega_K'$ the conjugate over $K$ of $\omega_K$. 

For $\O$ an order, denote by $\O^0$ its elements of null trace. If $x$ belongs to $\O$, we have $x=\frac{\mathrm{tr}(x)}{2}+\left(x-\frac{\mathrm{tr}(x)}{2}\right)$ and from this we observe that any $x\in\O$ can be written as $x=\frac{a+y}{2}$ with $a\in\Z_F$ and $y\in\O^0$. One can easily check that this decomposition of $x$ is unique. Write $\omega_K=\frac{t+\omega_{K,0}}{2}$ with $t\in \Z_F$ and $\omega_{K,0}\in\O_K^0$. Set $d_K\colonequals \omega_{K,0}^2\in \Z_F$, then the discriminant ideal $\mathfrak{d}_{\O_f/\Z_F}=f^2d_K\Z_F$. Observe that $K=F(\sqrt{d_K})$.

\subsection{Homogeneous toral sets}

Let $B$ be a quaternion algebra over $F$ and let $\GG$ be the algebraic group $PB^\times$ over $F$. We denote by $[\GG]$ the locally compact space $\lquot{\GG(F)}{\GG(\A_F)}$.

The Lie algebra $\mathfrak{g}$ of $\GG$ is identified with $B^0$, the space of quaternion of null trace. If $F[\varepsilon]$ denotes the dual numbers over $F$ i.e. $\varepsilon^2=0$, the previous claim can be seen from the exact sequence \[0\to B/F\to\GG(F[\varepsilon])\to\GG(F)\to1,\] where the first map is the one induced by $x\mapsto 1+x\varepsilon$. Note that the inclusion $B^0\subseteq B$ induces an isomorphism $B^0\cong B/F$.

Let $K/F$ be a quadratic extension and $\psi\colon K\to B$ an $F$-algebra embedding. Let $\TT_\psi$ be the image induced by $\psi$ of the algebraic torus $\TT_K\colonequals \mathrm{res}_{K/F}\mathbb{G}_m/\mathbb{G}_m$ in $\GG$. An homogeneous toral set in $[\GG]$ is the image of $\TT_\psi g$ in $[\GG]$ for some $\psi$ as before and $g\in \GG(\A_F)$. We denote this image by $[\TT_\psi g]$. More generally, we use $[\cdot]$ to denote the projection $\GG(\A_F)\to[\GG]$. Likewise, if $\KK$ is a compact subset of $\GG(\A_F)$, we use $[\GG(\A_F)]_{\KK}$ to denote the quotient space $[\GG]_{\KK}$ and $[\cdot]_{\KK}$ denotes the projection $[\GG]\to[\GG]_{\KK}$.

\subsection{The Discriminant.}

Take $\Lambda$ a lattice in $V\colonequals \mathfrak{g}\otimes\mathfrak{g}$. The group $\GG(F)=B^\times/F^\times$ acts on $V$ by the adjoint representation which under the isomorphism $V\cong B^0\otimes B^0$, can be seen as $g\cdot(a\otimes b)=gag^{-1}\otimes gbg^{-1}$ for any $g\in\GG(F)$ and $a,b\in B^0$. For $\p$ a prime in $\Z_F$, let $||\cdot||_\p$ be the norm in $V$ such that $\Lambda\otimes\O_\p$ is its unit ball.

Let $Y=[\TT_\psi g]$ be a homogeneous toral set. The Lie algebra $\mathfrak{t}$ of $\TT_\psi$ is identified with $\psi(K)\cap B^0=\psi(K^0)$. Let $w$ be an $F$-generator of $\mathfrak{t}$ and consider $\iota(\mathfrak{t})\colonequals\frac{w\otimes w}{C(w,w)}\in V$, where $C$ denotes the Cartan-Killing form on $\mathfrak{g}$. Under $\mathfrak{g}\cong B^0$, for $a\in B^0$ we have $C(a,a)=-4\mathrm{nr}(a)$.

\begin{definition}
The (finite part of the) discriminant of $Y$ (relative to $\Lambda$) is \[D_{\Lambda,f}\colonequals\prod_{\p}||g_\p^{-1}\cdot\iota(\mathfrak{t}) ||_\p.\]
\end{definition}

\begin{remark} There is also a component at infinity of the discriminant but we will make no use of it. This is because we will work with sequences of discriminants with constant discriminant at infinity. See \cite{ELMV3}.
\end{remark}

Now let $R$ be an $\Z_F$-order in $B$ and consider $R^T\colonequals (\Z_F+2R)\cap B^0$. It is a lattice in $B^0$. For $x$ in $\O$, write $x=\frac{a+y}{2}$ with $a\in\Z_F$ and $y\in\O^0$. Note that $\psi(x)\in R$ if and only if $\psi(y)\in R^T$.

Let $A$ be a commutative ring. Remember that an element $x$ in an $A$-module $M$ is primitive if $x\in a M$ with $a\in A$ implies $a\in A^\times$. %An element in an $\Z_F$-module $M$ is said to be $\p$-primitive if it is primitive when seen as an element in the $\O_{F,\p}$-module $M\otimes_{\Z_F}\O_{F,\p}$. Observe that $\p$-primitive for all prime $\p$ implies primitve for the original $\Z_F$-module $M$.

\begin{lemma}\label{opt} The following are equivalent:
\begin{enumerate}[font=\normalfont]
\item The embedding $\psi\colon K\to B$ is optimal with respect to $\Z_F$ i.e. $\psi(K)\cap R=\psi(\Z_F)$.
\item The element $\psi(f\omega_K)$ is primitive in $R$.
\item The element $\psi(f\omega_{K,0})$ is primitive in $R^T$.
\end{enumerate}
\end{lemma}
\begin{proof} First we prove that (1) is equivalent to (2). Let $a\in \Z_F$ and assume that $\psi(f\omega_K)\in aR$. Then, $\psi(f\omega_K)\in (aR)\cap\psi(K)=a(R\cap\psi(K))$ since $a\psi(K)=\psi(K)$. Therefore $\psi(f\omega_K)$ is primitive in $R$ if and only if it is primitive in $\psi(K)\cap R$. Both (1) and (2) implies that $\psi(\O_f)\subseteq R$ so we will assume this contention. Now $\psi(\O_f)\subseteq \psi(K)\cap R$ so $\psi(K)\cap R$ is certainly an order $\psi(\O_{f'})\subseteq \psi(\O_K)$ containing $\psi(\O_f)$. Then $\psi(\O_f)=\psi(\O_{f'})$ if and only if $f=f'$ and this is equivalent to $\psi(f\omega_K)$ being primitive in $\psi(\O_{f'})$.

Now we prove that (2) is equivalent to (3). Let $a\in \Z_F$. Remember that $f\omega_K=\frac{ft+f\omega_{K,0}}{2}$. Now we have $\psi(f\omega_K)\in aR$ if and only if $\psi(f\omega_{K,0})\in (\Z_F+2aR)^0$. We conclude the proof if we show that $(\Z_F+2aR)^0=aR^T$. We have $R^T\subseteq (\Z_F+2aR)^0$ so it is enough to prove the opposite containment. Take $x\in(\Z_F+2aR)^0$ then $x=b+2ar$ with $b\in\Z_F$, $r\in R$ and $b+a\mathrm{tr}(r)=0$. But $\mathrm{tr}(r)\in\Z_F$ so $b\in a\Z_F$ and we have finished.
%We know that $\psi(c\beta)$ belongs to $R^T$ if given the optimality or tautologically if it is a primitive element as claimed. The optimality condition is equivalent to $\psi(K_\p)\cap R_\p=\psi(\O_\p)$ for every $\p$. So it is sufficient to prove that this local condition is equivalent to $\psi(c\beta)$ being $\p$-primitive. First assume that $\p\nmid 2$, since $2\O_K\subseteq\Z_F[\beta]$ we have that $\O_{K,\p}=\O_{F,\p}[\beta]$. Therefore $\O_\p=\O_{F,\p}+c\O_{K,\p}=\O_{F,\p}+c\beta\O_{F,\p}$. Observe now that $\psi(c\beta)$ belongs to $\psi(K_\p)\cap R^T_\p$ but $R^T_\p=R_\p^0$ so it actually belongs to the $\O_{F,\p}$-submodule $\psi(K_\p)\cap R_\p^0=\psi(c\beta\O_{F,\p})$ of $R^T_\p$. This module is free of rank $1$ and so $\p$-primitive means being a generator of it. This is evident.
\end{proof}

Since $g_\p$ belongs to $R_\p^\times/\O_{F,\p}^\times$ for all but finitely many primes $\p$, we have $g_{\p}R_{\p} g_{\p}^{-1}=R_\p$ almost everywhere. Therefore the intersección $B\cap\prod_{\p}g_\p R_\p g_\p^{-1}$ defines a global order in $B$ that we denoted as $R_Y$. Define $\O_Y$ to be the global order in $\psi(K)$ given by $\psi(K)\cap R_Y$.

Denote by $|\cdot|$ the usual absolute value in $\R$.

\begin{proposition}\label{discriminant} Let $\Lambda$ be the lattice $R^T\otimes R^T$ in $V=B^0\otimes B^0$ and $Y$ a homogeneous toral set in $[\GG]$. Then, up to a multiplicative constant (depending only on $F$), $D_{\Lambda,f}(Y)$ equals $|N_{F/\Q}(\mathfrak{d}_{\O_Y/\Z_F})|$.
\end{proposition}

\begin{proof} Write $\O_Y=\psi(\O_f)$. By Lemma \ref{opt} we have that $\psi(f\omega_{K,0})$ is primitive in $R_Y^T$. This is equivalent to $\psi(f\omega_{K,0})$ being primitive in the $\O_{F,\p}$-module $(R_Y^T)_\p$ for every prime $\p$. Conjugating by $g_\p^{-1}$ we have that $g_\p^{-1}\psi(f\omega_{K,0})g_\p$ is primitive in $(R^T)_\p$ for every prime $\p$. Now consider $w=\psi(f\omega_{K,0})$ as an $F$-generator of $\mathfrak{t}$ in the definition of $D_{\Lambda,f}$. Since $g_\p^{-1}\psi(f\omega_{K,0})g_\p$ is primitive in $(R^T)_\p$, the element $g_\p^{-1}\cdot(w\otimes w)$ is primitive in $\Lambda_\p$ and so by definition of $||\cdot||_\p$ we have that $||g_\p^{-1}\cdot(w\otimes w)||_\p=1$. Applying the definition of $D_{\Lambda,f}$ and using the product formula on $F$ we have \begin{align*}D_{\Lambda,f}(Y)&=\prod_\p {|-4\mathrm{nr}(w)|}_{\p}^{-1}\\&=\prod_{\sigma\colon F\to \R}|\sigma(-4\mathrm{nr}(w))|\\&=|N_{F/\Q}(-4\mathrm{nr}(w))|.\end{align*} But $\mathrm{nr}(w)=N_{K/F}(f\omega_{K,0})=-f^2d_K$ since $\omega_{K,0}\in\O_K^0$. We conclude that \[D_{\Lambda,f}(Y)=4^{[F\colon\Q]}|N_{F/\Q}(f^2d_K)|.\] On the other hand, the discriminant ideal $\mathfrak{d}_{\O_Y/\Z_F}$ is generated by $f^2d_K$ and this finishes the proof. \end{proof}

%For a field extension $L/M$ we denote by $\mathfrak{d}_{L/M}$ the relative ideal discriminant.. On the other hand, we have that $\mathfrak{d}_{K/F}$ is the principal ideal in $\Z_F$ generated by $d_{K/F}(\omega)=(\omega-\omega')^2=\beta^2$, $\mathfrak{d}_{K/\Q}$ is the principal ideal in $\Z$ generated by the discriminant $d_{K/\Q}=d_{\O_K/\Z}=(d_{\O_Y/\Z})/d^2$ and $\mathfrak{d}_{F/\Q}$ is a principal ideal in $\Z$, say generated by $n_F$. Hence, the formula $\mathfrak{d}_{K/\Q}=\mathfrak{d}_{F/\Q}^2N_{F/\Q}(\mathfrak{d}_{K/F})$ implies that $d_{\O_Y/\Z}=\pm d^2n_F^2N_{F/\Q}(\beta^2)$

%\begin{remark} Note that $N_{F/\Q}(\beta^2)$ is a constant depending only on $K$. Indeed $\beta$ was defined by the equation $\omega=\frac{a+\beta}{2}$ with $\O_K=\Z_F[\omega]$, $a\in\Z_F$ and $\beta\in\O_K^0$. Also note that this decomposition of $\omega$ is unique. If $\gamma$ is another generator of $\O_K$ over $\Z_F$ we must have $\gamma=u+v\omega=\frac{(2u+va)+(v\beta)}{2}$ with $u,v\in\Z_F$ and we can see that $2u+va\in\Z_F$ and $v\beta\in\O_K^0$. But since $\omega$ can also be written in terms of $\gamma$ we must have $v\in\Z_F^\times$. Therefore any other ``$\beta$'' must be of the form $v\beta$ with $v\in\Z_F^\times$. We have $N_{F/\Q}(u^2\beta^2)=N_{F/\Q}(\beta^2)$. This remark and the previous proposition shows that $D_{\Lambda,f}(Y)\to\infty$ if and only $N_{F/\Q}(\beta^2)\to\infty$ or the conductor of $\O_Y$ goes to infinity.
%\end{remark}

\subsection{Equidistribution}

%\begin{theorem} Let $\{Y_i\}$ be a sequence of homogeneous toral sets whose discriminant approaches $\infty$ with $i\to\infty$. Then any weak* limit of the measures $\mu_{Y_i}$ is a homogeneous probability measure on $[\GG]$, invariant under $\GG(\A_F)^+$.
%\end{theorem}
%\begin{proof} See Theorem 4.6 in \cite{ELMV3}.
%\end{proof}

Our hypothesis on the narrow class number of $F$ implies $\A_{F,f}^\times=F_+^\times\widehat{\Z_F}^\times$ and the exactness of the sequence \begin{equation*}\label{signs}1\to\Z_F^\times/\Z_{F,+}^\times\to\prod_{\sigma\colon F\to\R}\R^\times/\R_{>0}\to1.\end{equation*} Since $F$ is totally real this shows the equality $\Z_{F,+}^\times=(\Z_{F}^\times)^2$ since the squared units are totally positive and both subgroups have the same index in $\Z_F^\times$. From this, we have \[\lrquot{F^\times}{\A_F^\times}{(\A_F^\times)^2}\cong\lrquot{F_+^\times}{\A_{F,f}^\times}{(\A_{F,f}^\times)^2}\cong\lrquot{\Z_{F,+}^\times}{\widehat{\Z_{F}}^\times}{(\widehat{\Z_F}^\times)^2}\cong\rquot{\widehat{\Z_{F}}^\times}{(\widehat{\Z_F}^\times)^2}.\]

Let $\KK$ be a compact subgroup of $\GG(\A_F)$ and assume that $\mathrm{nrd}(\KK)$ has finite index in $F^\times\backslash\A_F^\times/(\A_F^\times)^2$. We define $S(\KK)$ to be the finite set of primes in $\Z_F$ for which the natural map \[\lrquot{F^\times}{\A_F^\times}{(\A_F^\times)^2\mathrm{nrd}(\KK)}\cong\prod_{\p\in S(\KK)}\Z_{F,\p}^\times/(\Z_{F,\p}^\times)^2\] is an isomorphism.

Denote by $\mu$ the unique probability measure in $[\GG]$ coming from a Haar measure in $\GG(\A_F)$ and a counting measure in $\GG(F)$.

\begin{theorem}\label{equidistribution} Let $\KK$ be a compact subgroup of $\GG(\A_F)$ and asssume that $\mathrm{nrd}(\KK)$ has finite index in $F^\times\backslash\A_F^\times/(\A_F^\times)^2$. Let $\{Y_i\}$ be a sequence of homogeneous toral sets whose discriminant approaches $\infty$ when $i\to\infty$ and whose associated quadratic fields $K_i$ do not ramify at primes $\p\in S(\KK)$. Then, the sequence of measures $({[\cdot]_{\KK}})_*(\mu_{Y_i})$ converges weak-* to  $({[\cdot]_{\KK}})_*(\mu)$ on $[\GG]_{\KK}$.
\end{theorem}
\begin{proof} This is a refinement of Theorem 4.6 in \cite{ELMV3}. One wants to prove that every continuous compactly supported function $f$, invariant under right multiplication by $\KK$, we have \begin{equation}\label{test}\lim_{i\to\infty}\int_{[\GG]}fd\mu_{Y_i}=\int_{[\GG]}fd\mu.\end{equation} Let $\widetilde{\GG}$ denote the algebraic group $B^1$ of quaternions of norm $1$ and denote by  $\GG(\A_F)^+$ the image of the natural map $\widetilde{\GG}(\A_F)\to\GG(\A_F)$. The product $\GG(F)\GG(\A_F)^+$ is a closed subgroup. Indeed, $\GG(\A_F)^+$ is normal and Theorem 5.24 in \cite{AGNT} implies that $\lquot{\widetilde{\GG}(F)}{\widetilde{\GG}(\A_F)}$ has finite volume. The pushforward of this measure to $\lquot{\GG(F)}{\GG(F)\GG(\A_F)^+}$ is a finite $\GG(\A_F)^+$-invariant measure and then $\GG(F)\cap\GG(\A_F)^+$ is a lattice in $\GG(\A_F)^+$. Now Theorem 1.13 in \cite{Raghu} implies that $\GG(F)\GG(\A_F)^+$ is closed. Assume that $\mu_{Y_i}$ converges weak-* to a measure $\mu_{\infty}$. Theorem 4.6 in \cite{ELMV3} implies that $\mu_\infty$ is a homogeneous probability measure invariant under $\GG(\A_F)^+$. Since $\GG(F)\GG(\A_F)^+$ is closed, the discussion made in section 9 of \cite{Aka} applies and we can reduced to prove \eqref{test} for functions which are right $\GG(\A_F)^+\KK$-invariant. These are functions over the space \[\lrquot{\GG(F)}{\GG(\A_F)}{\GG(\A_F)^+\KK}.\] The reduced norm $\mathrm{nrd}$ induces an homeomorphism between $\lrquot{\GG(F)}{\GG(\A_F)}{\GG(\A_F)^+}$ and the compact abelian group $\lrquot{F^\times}{\A_F^\times}{(\A_F^\times)^2}$. By the Weyl criterion on characters, we are then reduced to work with $f$ being of the form $\chi\circ\mathrm{nrd}$, where $\chi$ is a non-trivial quadratic Hecke character over $F$ trivial over $\mathrm{nrd}(\KK)$. By Class Field Theory, to each such $\chi$ there's attached a quadratic extension $K_\chi/F$ that ramifies only at $\p\in S(\KK)$ (see section 1.1.1 in \cite{Hida}). Note that in this case, the RHS of $\eqref{test}$ is $0$ since we are integrating a character over the whole group. On the other hand, since $K_\chi$ can not be one of the $K_i$'s (comparing ramification), there exists some prime $\mathfrak{l}_i$ of $F$ inert in $K_{\chi}$ but split in $K_i$. Denote by $\psi_i\colon K_i\to B$ the embedding attached to $Y_i$ and let $t_i\in \A_{K_i}$ be the usual adelic element attached to a prime $\mathfrak{L}_i\mid \mathfrak{l}_i$ in $K_i$. Then $\chi(\mathrm{nrd}(\psi_i(t_i)))=\chi(N_{K_i/F}(\mathfrak{L_i}))=\chi(\mathfrak{l})=-1$. Now using the change of variables $t\mapsto tt_i$, we can show that the sequence in the LHS of \eqref{test} satifies \[\int_{Y_i}\chi(\mathrm{nrd}(t))d\mu_{Y_i}(t)=\chi(g_i)\int_{\TT_{K_i}}\chi(N_{K/F}(t))d\mu_{\TT_{K_i}}(t)=-\int_{Y_i}\chi(\mathrm{nrd}(t))d\mu_{Y_i}(t),\] finishing the proof. \end{proof}

\section{Optimal embeddings for the algebra of matrices}\label{embeddings}

We use the embedding $(\sigma_0,...,\sigma_r)\colon M_2(F)\to M_2(\R)^{r+1}$ to identify $\PGL_2(F_\infty)$ with $\PGL_2(\R)^{r+1}$. Let $\GG$ be the algebraic group $\mathrm{PGL} _2$ defined over $F$. Let $\O$ be an order in $K$. 

Let $S$ be a finite set of places of $F$ containing the infinite ones. We use $\Z_F[S^{-1}]$ to denote the localization of $\Z_F$ at the finite primes of $S$. If $\O$ is an order in $K$, we set $\O[S^{-1}]=\O\otimes_{\Z_F}\Z_F[S^{-1}]$.

We say that $\psi\colon K\to B$ is an $S$-optimal embedding with respect to $\O$ (and $M_2(\Z_F)$) if $\psi(K)\cap M_2(\Z_F[S^{-1}])=\psi(\O[S^{-1}])$. This is equivalent to $\psi(K_\p)\cap M_2(\Z_{F,\p})=\psi(\O_{\p})$ for every finite place $\p\not\in S$.

We denote by $opt(\O[S^{-1}])$ the collection of such optimal embeddings. The symbol $[\psi]$ denotes the $\PGL^+_2(\Z_F[S^{-1}])$-conjugacy class of $\psi$ and accordingly $[opt(\O[S^{-1}])]$ is the collection of all such classes.

Let $\KK^S=\prod_{\p\not\in S}\KK_\p$ be the subgroup of $\GG(\A^S)$ defined by $\KK_\p=\mathrm{PGL}_2(\Z_{F,\p})$. Since $F$ has class number $1$, there is a local-global principle for lattices in $F^2$ so $\GG(\A_{F,f})=\GG(F)\GG(\widehat{\Z}_F)$. In particular we have a decomposition $\GG(\A^S)=\GG(F)\GG(F_S)\KK^S$. Since $F$ has narrow class number $1$, for each $0\leq j\leq r$, there exists an element $\gamma_j\in \GG(F)$ such that $\sigma_j(\det(\gamma_j))<0$ and $\sigma_i(\det(\gamma_j))>0$ if $i\neq j$. This implies that we have an identification \begin{equation}\label{doublequot}\lrquot{\GG(F)}{\GG(\A_F)}{\KK^S}\cong\lquot{\mathrm{PGL}_2(\Z_F[S^{-1}])}{\PGL_2(F_S)}\cong \lquot{\PGL^+_2(\Z_F[S^{-1}])}{\PGL^+_2(F_{S})}.\end{equation} It is given by sending the class of $x\in \GG(\A_F)$ to the $\PGL^+_2(\Z_F[S^{-1}])$-class of $\gamma x_S$, where $\gamma\in\GG(F)$ is such that $\gamma x^S\in\KK^S$ and $\gamma x_\infty\in\PGL^+_2(\R)^{r+1}$.

Take $\psi\in opt(\O)$. Remember that $\psi$ induces another embedding $\psi\colon \TT_K\to\GG$. Let $\TT_K(\Z_F^S)\colonequals \psi^{-1}(\TT_\psi\cap \KK^S))$. Since $\KK^S$ is the projective image of $M_2(\Z_F^S)^\times$ in $\GG(\A_{F}^S)$, the optimality condition induces an embedding \begin{equation}\label{injopt}\lrquot{K^\times}{\A_{K}^\times}{{(\O^{S})}^\times}\cong\lrquot{\TT_K(F)}{\TT_K(\A_F)}{\TT_K({\Z_F^S})}\xrightarrow{\psi}\lrquot{\GG(\A)}{\GG(\A_F)}{\KK^S}.\end{equation}

Take $t\in \A_{K}^\times$. Let $\gamma \in \GG(F)$ such that $\gamma\psi(t^S,t_\infty)\in \KK^S\times\PGL^+_2(\R)^{r+1}$. We define $t\star\psi$ as $\gamma\psi\gamma^{-1}$. Denote by $Pic^+(\O[S^{-1}])$ the narrow class group of $\O[S^{-1}]$. We have an identification \[Pic^+(\O[S^{-1}])\cong \lrquot{K^\times}{\A_K^\times}{{(\O^S)}^\times K_{S_f}^\times (K_\infty^\times)^2 },\] where $S_f$ denote the finite places of $S$.

%Take $t\in \A_{K,f}^\times$. Denote by $t\cdot_\psi R$ the unique order of $B$ whose local component at $\p$ is $\psi(t_\p)M_2(\O_{F,\p})\psi(t_\p)^{-1}$. There exist $\gamma\in\GG(F)$ such that $t\cdot_\psi R=\gamma^{-1}R\gamma$. Remember the identification $Pic(\O)\cong\lrquot{K^\times}{\A_{K,f}}{\widehat{\O}^\times}$. Assume that $\a\in Pic(\O)$ is identified with the class of $t\in \A_{K,f}^\times$ under the previous isomorphism. We define $\a\star \psi=\gamma\psi\gamma^{-1}$.

\begin{proposition}\label{action} The operation $\star$ induces a well defined action of $Pic^+(\O[S^{-1}])$ on $[opt(\O[S^{-1}])]$. Moreover, this action is simple and transitive.
\end{proposition}
\begin{proof} The proof of this proposition is analogue to the proof of Lemma 4.2 and Proposition 4.3 in \cite{paper1}. The action is well defined since \eqref{injopt} is an embedding and the fact that conjugating $\psi$ by an element $\delta$ in $\PGL_2^+(\Z_F[S^{-1}])$ changes $\gamma$ by $\gamma\delta^{-1}$ since $\delta\in \KK^S\times \PGL_2^+(\R)^{r+1}$. Thus, after conjugating $\delta\psi\delta^{-1}$ by $\gamma\delta^{-1}$ we obtain $\gamma\psi \gamma^{-1}$ again. 

Since $\gamma\in \GG(F)=B^\times/F^\times$, the map $\gamma\psi \gamma^{-1}$ is indeed a new embedding that send $K\to B$. For the optimality it is enough to show that for $\p\notin S$ one has \[\gamma\psi(K_\p)\gamma^{-1}\cap M_2(\Z_{F,\p})=\gamma\psi(\O_\p)\gamma^{-1}.\] We know $\psi(K_p)\cap M_2(\Z_{F,\p})=\psi(\O_\p)$ but since $\mathrm{im}(\psi)$ is commutative, after conjugating by $\gamma\psi(t_\p)$ we have \[\gamma\psi(K_\p)\gamma^{-1}\cap (\gamma\psi(t_\p))M_2(\Z_{F,\p})(\gamma\psi(t_\p))^{-1}=\gamma\psi(\O_\p)\gamma^{-1}\] and now we use $(\gamma\psi(t_\p))M_2(\Z_{F,\p})(\gamma\psi(t_\p))^{-1}=M_2(\Z_{F,\p})$ since $\gamma\psi(t_\p)\in \PGL_2(\Z_{F,\p})$.

Consider elements $t$ and $s$ in $\A_K^\times$ and suppose that $\gamma,\delta\in \GG(F)$ are such that $\gamma\psi(s^S), \delta\gamma\psi\gamma^{-1}(t^S,t_\infty)\in\KK^S\times\PGL_2^+(\R)^{r+1}$. The equality $\delta\gamma\psi(t^Ss^S,t_\infty s_\infty)=\delta\gamma\psi(t^S,t_\infty)\gamma^{-1}\gamma\psi(s^S,s_\infty)\in\KK^S\times\PGL_2^+(\R)^{r+1}$ shows that $[ts\star \psi]=[t\star(s\star\psi)]$.

Now we show that the action is transitive. Let $\psi$ and $\psi'$ be two $S$-optimal embeddings for $\O$. As a consequence of Skolem-Noether there exists $\gamma\in \GG(F)$ such that $\psi'=\gamma\psi \gamma^{-1}$. To prove transitivity we need to show that there exists an adelic element $t$ such that $\gamma\psi(t^S,t_\infty)\in \KK^S\times\PGL_2^+(\R)^{r+1}$. Since both $\psi$ and $\psi'$ are $S$-optimal for $\O$, we have for $\p\notin S$ \begin{equation}\label{stab}\psi(\O_\p)=\psi(K_\p)\cap M_2(\Z_{F,\p})=\psi(K_\p)\cap \gamma^{-1}M_2(\Z_{F,\p})\gamma.\end{equation} This relation implies that the orders $M_2(\Z_{F,\p})$ and  $\gamma^{-1}M_2(\Z_{F,\p})\gamma$ (seen as element in the Bruhat-Tits tree of $\PGL_2(F_\p)$) belong to the same $\psi(K_\p)^\times$-orbit under conjugation. This is because Corollary 2.2 in \cite{paper1} and its proof are still valid when replacing $\Q_\ell$ by any local field and $\ell$ (in the computations of the proof) by its residue characteristic. In conclusion, there exists some element $t_\p$ in $K_\p^\times$ such that $\psi(t_\p)M_2(\Z_{F,\p})\psi(t_{\p})^{-1}=\gamma^{-1}M_2(\Z_{F,\p})\gamma$ and this is equivalent to $\gamma\psi(t_\p)\in\KK_\p$. Once we have $t\in \A_K^\times$ such that $\gamma\psi(t^S)\in\KK^S$, the component $t_\infty$ can be easily be constructed according to the signs of $\gamma$ to ensure that $\gamma\psi(t^S,t_\infty)\in \KK^S\times\PGL_2^+(\R)^{r+1}$.

Finally, we give the proof of the simplicity of the action. Assume that $t$ acts trivially over $[\psi]$. There exists some $\delta\in PGL_2^+(\Z_F[S^{-1}])$ such that $\gamma\psi\gamma^{-1}=\delta\psi\delta^{-1}$, then there exists some $s\in K^\times$ such that $\delta^{-1}\gamma=\psi(s)$. Therefore $\psi((st)^S,(st)_\infty)\in \KK^S\times\PGL_2^+(\R)^{r+1}$ and this says that $st\in{(\O^S)}^\times K^\times_{S_f}\prod_{\sigma\colon F\to\R}\pm(K_\sigma^\times)^2$. Now this shows that $t$ is trivial on $\lrquot{K^\times}{\A_K^\times}{{(\O^S)}^\times K_{S_f}^\times (K_\infty^\times)^2 }$. 
\end{proof}

\section{Equidistribution of ATR cycles}\label{ATR}

\subsection{Parametrization of ATR cycles}\label{ATRcyclespar}

Let $K/F$ be an ATR extension and $\O$ an order in $K$. Take $\psi\in opt(\O)$ and consider $\Delta_\psi$ its associated ATR cycle as in the Introduction. If $\tau_0\in\H_0$ and $\mathcal{Y}_j\subseteq\H_j$ ($1\leq j\leq r$) denotes the unique point and geodesics stable under $\psi(K)^\times$ respectively, $\Delta_\psi$ is the projection of $\{\tau_0\}\times\prod_j\mathcal{Y}_j$ in the Hilbert modular variety $\PSL_2(\Z_F)\backslash\H^{r+1}$. Observe that $\Delta_\psi$ depends actually only on $[\psi]$.

The group $\psi(K_{\sigma_j}^\times)$ acts transitively on $\mathcal{Y}_j $. Denote by $\Gamma_{x}$ the stabilizer in $\PSL_2(\Z_F)$ of $x$ in some $\H_j$ or in $\H^{r+1}$. Note that $\Gamma_{\tau_0}=\psi(K_{\sigma_0})^\times\cap \PSL_2(\Z_F)$. Since $\psi$ is optimal with respect to $\O$, we obtain that $\Gamma_{\tau_0}=\psi(\O_{+}^\times)/{\Z_F}^{\times}$. By Dirichlet's unit theorem, $\Gamma_{\tau_0}$ has rank $(2r+1)-1-r=r$.

Note that we obtain a uniformization \[\Delta_\psi\cong\{\tau_0\}\times\lquot{\Gamma_{\tau_0}}{\prod_{j=1}^r\mathcal{Y}_j}\cong \lrquot{\O_{+}^\times}{K_\infty^\times}{{F_\infty^\times}},\] which shows that $\Delta_\psi$ is compact. Use a Haar measure in $K_\infty^\times$ to endow $\lrquot{\O_{+}^\times}{K_\infty^\times}{{F_\infty^\times}}$ with a finite measure. The measure $\nu_\psi$ from the introduction is the push-forward of this measure to $\Delta_\psi$. By Proposition \ref{action} with $S$ equal to the infinite places, there are exactly $h_+(\O)=Pic^+(\O)$ ATR cycles of conductor $\O$ in  $\PSL_2(\Z_F)\backslash \H^{r+1}$. %As in the introduction, $\mu_\O$ denotes the unique probability measure proportional to $\sum_{[\psi]}\nu_{\psi}$.

%\begin{theorem}\label{equidistributionATR'} The collection of $h_+(\Z_F)$ ATR cycles attached to $\Z_F$ becomes equidistributed on $\Gamma\backslash \H^{r+1}$ as long as $|N_{F/\Q}(d_K)|$ or the conductor $f$ goes to $\infty$. In other words, the measures $\mu_{\O}$ converges weak-* to the mesaure $\mu$.
%\end{theorem}

\subsection{Equidistribution}\label{equidATR}

Remember that $\PGL_2^+(\R)$ can be identified with $T^1\H$, the unit tangent bundle of $\H$. This is done by defining the action $g(z,\zeta)\colonequals(gz,Dg(z)\zeta)$ for any $(z,\zeta)\in \H\times T_z^1\H$, where $Dg(z)=(cz+d)^{-2}$ if $g=\begin{pmatrix}a&b\\c&d\end{pmatrix}$. Each geodesic in $\H$ can be lifted into an orbit under the geodesic flow in $T^1\H$. Under the previous identification this correspond to a right $A^+(\R)$-orbit where $A$ denote the diagonal subgroup of $\PGL_2$.

Let $\KK=\KK_f\times\prod_{\sigma_j}\KK_j$ be the subgroup of $\GG(\A_F)$ defined by $\KK_f=\mathrm{PGL}_2(\widehat{\Z_F})$, $K_0=\mathrm{PSO}_2(\R)$ and $K_j$ trivial for $1\leq j\leq r$.

As in \eqref{doublequot}, we have an identification \[\lrquot{\GG(F)}{\GG(\A_F)}{\KK}\cong\lquot{\mathrm{PSL}_2(\Z_F)}{\left(\H_0\times\prod_{j=1}^rT^1\H_j\right)}\]

For $\psi\in opt(\O)$, we lift the ATR cycle attached to it as closed $A^+(\R)^r$-orbit in the following fashion. By the Skolem-Noether theorem, there exists $g^\psi\in\GG(F_\infty)$ such that $g^\psi_{\sigma_0}\PSO_2(\R){g_{\sigma_0}^\psi}^{-1}=\psi_0(K_{\sigma_0})^\times$ and $g^\psi_{\sigma_j}A(\R){g_{\sigma_j}^\psi}^{-1}=\psi_0(K_{\sigma_j})^\times$. The right orbits $g^\psi_{\sigma_0}\PSO_2(\R)$ and $g^\psi_{\sigma_j}A^+(\R)$ are fixed if we normalize the choice of $g^\psi$ to satisfy $g^\psi_{\sigma_0}\in\PGL_2^+(\R)$ (implying that $g^\psi_{\sigma_0}\cdot i=\tau_0$) and that $g^\psi_{\sigma_j}\cdot\infty=\sigma_j(\tau_0)$. If this is the case, the $\PSL_2(\Z_F)$-orbit of $(g^\psi_{\sigma_0}\PSO_2(\R),g^\psi_{\sigma_1}A^+(\R),...,g^\psi_{\sigma_r}A^+(\R))\subseteq \PGL_2^+(\R)^{r+1}$ projects onto the ATR cycle $\Delta_\psi$.

\begin{proposition}\label{propA} Fix $\psi_0\in opt(\O)$. The projection of the homogeneous toral set $[\TT_{\psi_0}g]$ to $[\GG]_{\KK}$ is the collection of $h^+(\O)$ lifts of the ATR cycles.
\end{proposition}

\begin{proof} Denote $h^+=h^+(\O)$ and let  $t_1,...,t_{h^+}\in \A_K^\times$ be adelic representatives $Pic^+(\O)$. That is, we have \[\A_K^\times=\bigsqcup_{i=1}^{h^+}K^\times t_i\widehat{\O}^\times(K_\infty^\times)^2.\] The image of $\TT_{\psi_0}$ in $[\GG]_{\KK}$ is formed by the $\GG(F)$-orbits of the sets \[\psi_0(t_{i,f})\KK_f\times \psi_0(t_{i,\infty}(K_\infty^\times)^2)\KK_0.\] Under the identification \eqref{doublequot}, in $\lquot{\PSL_2(\Z_F)}{\PGL_2^+(\R)}^{r+1}$ they correspond to the $\PSL_2(\Z_F)$-orbits of the sets $\gamma_i\psi(t_{i,\infty}(K_\infty^\times)^2)$, where $\gamma_i\psi_0(t_i)\in \KK_f\PGL_2^+(\R)^{r+1}$.  Now we apply the right translation by $g^\psi$, and obtain that the  image of $[\TT_{\psi_0}g^\psi]$ in the quotient space $\lquot{\PSL_2(\Z_F)}{\PGL_2^+(\R)}^{r+1}$ is composed of the $\PSL_2(\Z_F)$-orbits of the sets \[\gamma_i\psi_0\gamma_i^{-1}((K_\infty^\times))^2\gamma_i\psi_0(t_{i,\infty})g^\psi.\] By Proposition \ref{action}, the $\gamma_i\psi\gamma_i$ formed a complete set of representatives for $[opt(\O)]$. Since $K_{\sigma_0}\cong\C$, $\psi_0(t_{\sigma_0})\in \PGL_2^+(\R)$ and then $\gamma$ has positive sign at $\sigma_0$ implying that $\gamma_ig^\psi_{\sigma_0}\cdot i$ is the unique fixed point of $\gamma_i\psi_0\gamma_i^{-1}$ in $\H_0$. The condition at $1\leq j\leq r$ for $\gamma_ig^\psi_{\sigma_j}\cdot \infty$ follows immediately. This finishes the proof of the first statement because then, \[\gamma_i\psi_0\gamma_i^{-1}((K_\infty^\times))^2\gamma_i\psi_0(t_{i,\infty})g^\psi=(g_{\sigma_0}^{\gamma_i\psi_0\gamma_i^{-1}}\mathrm{PSL}_2(\R),g_{\sigma_1}^{\gamma_i\psi_0\gamma_i^{-1}}A^+(\R),...,g_{\sigma_r}^{\gamma_i\psi_0\gamma_i^{-1}}A^+(\R)).\]
\end{proof}

\begin{proof}[Proof of Theorem \ref{equidistributionATR}]
In view of Proposition \ref{propA}, it is enough to apply Theorem \ref{equidistribution} to the collection of tori $\{Y_\O=[\TT_{\psi_0}g]\}_{\O}$. We need to observe that our compact $\KK$ satisfies $S(\KK)=\emptyset$ so there is equidistribution as long as we prove that the discriminant of $\TT_{\psi}$ is $|N_{F/\Q}(\mathfrak{d}_{\O/\Z_F})|$. This follows from Proposition \ref{discriminant} since $g^{\psi_0}_\p$ is trivial for every $\p$ and therefore $\O_{Y_\O}=\O$. \end{proof}

\section{Equidistribution of Stark-Heegner cycles}\label{StarkH}
\subsection{Paramatrization of Stark-Heegner cycles} Let $K/\Q$ be a real quadratic extension inert at $p$ and $\O$ an order in $K$. Take $\psi\in opt(\O[1/p])$ and consider $\Delta_\psi$ its associated Stark-Heegner cycle as in the Introduction. Observe that $\Delta_\psi$ depends only on $[\psi]$.

Let $S=\{p,\infty\}$. The group $\psi(K_\infty^\times)$ acts transitively on $\mathcal{Y}_\psi$. Denote by $\Gamma_\tau$ the stabilizer of $\tau$ in $\PGL_2^+(\Z[1/p])$. Note that $\Gamma_{\tau_\psi}=\psi(K)^\times\cap\PGL^+_2(\Z[1/p])$. Since $\psi$ is $S$-optimal with respect to $\O$, we obtain that $\Gamma_{\tau_\psi}=\psi(\O[S^{-1}]^\times)/\pm p^\Z$.

We have the following uniformization \begin{equation}\label{unifSH}\Delta_\psi=\Gamma_{\tau_\psi}\backslash\mathcal{Y}_\psi\times \{\tau_\psi\}=\lrquot{\O[S^{-1}]^\times_+}{K_S^\times}{\Q_S^\times}\end{equation} which shows that $\Delta_\psi$ is compact by Dirichlet's unit Theorem for $S$-units. Use a Haar measure in $K_S^\times$ to endow $\lrquot{\O[S^{-1}]^\times_+}{K_S^\times}{\Q_S^\times}$ with a finite measure. The measure $\nu_\psi$ from the Introduction is the push-forward of this measure to $\Delta_\psi$ under \eqref{unifSH}. By Proposition \ref{action}, there are exactly $h_+(\O[1/p])=\# Pic^+(\O[1/p])$ Stark-Heegner cycles of conductor $\O$ in $\PGL^+_2(\Z[1/p])\backslash(\H\times\H_{p^2}^{unr})$.

\subsection{Equidistribution} Fix $u\in \Z_{p}^\times$ a non-square. Now fix $\alpha\in \Z_{p^2}$ a square-root of $u$. We have $\H_{p^2}^{unr}=\Q_p+\alpha\Q_p^\times$ and so $\PGL_2(\Q_p)$ acts transitively on $\H_{p^2}^{unr}$ as it is shown by the equality $\begin{pmatrix}y&x\\0&1\end{pmatrix}\cdot \alpha=x+y\alpha$. We denote by $\KK_p$ the stabilizer of $\alpha$ in $\PGL_2(\Q_p)$. Let $\mathcal{K}$ denote the quadratic subalgebra of $M_2(\Q_p)$ generated by $\begin{pmatrix}0&u\\1&0\end{pmatrix}$ over $\Q_p$. Then $\mathcal{K}$ is isomorphic to $\Q_{p^2}$ and $\KK_p$ is isomorphic to $\mathcal{K}^\times/\Q_p^\times$. This shows that $\KK_p$ is compact since it is homeomorphic to $\P^1(\Q_p)$. When compared to the case $S=\{\infty\}$, it plays the role of the orthogonal group $\mathrm{PSO}_2(\R)$.

Let $\KK=\KK^S\times\KK_p$ be the compact subgroup of $\GG(\A_{\Q,f})$ defined by $\KK^S=\PGL_2(\Z^S)$. As in \eqref{doublequot}, we have an identification \[\lrquot{\GG(F)}{\GG(\A_F)}{\KK}\cong\lquot{\mathrm{PGL}^+_2(\Z[1/p])}{(T^1\H\times\H_{p^2}^{unr})}.\]

As we did in section \ref{equidATR} with the geodesic component of an ATR cycle, for $\psi$ in $opt(\O[1/p])$, we lift the associated Stark-Heegner cycle as an $A^+(\R)$-orbit. Indeed, let $g^\psi\in\PGL_2^+(\Q_S)$ be such that $g^\psi_p\KK_p {g^\psi_p}^{-1}=\psi(K_p)^\times$ and $g^\psi_\infty A(\R){g^\psi_{\infty}}^{-1}=\psi(K_\infty)^\times$. The existence of this element in guaranteed by the Skolem-Noether Theorem. The orbits $g^\psi_p\KK_p=\psi(K_p)^\times g^\psi_p$ and $g^\psi_\infty A^+(\R)=\psi(K_\infty^\times)g^\psi_\infty$ are fixed if we normalize the choice of $g^\psi$ to satisfy $g^\psi_\infty\cdot\infty=\tau_\psi$ in $\R$ and $g^\psi_p\cdot\alpha=\tau_\psi$ in $\H_{p^2}^{unr}$. These choices imply that the $\PGL_2^+(\Z[1/p])$-orbit of $(g^\psi_\infty A^+(\R),g^\psi_p\KK_p)\subseteq \PGL_2(\Q_S)$ projects onto the Stark-Heegner cycle $\Delta_\psi$.

\begin{proposition} Fix $\psi_0\in opt(\O[1/p])$. The projection of the homogeneous toral set $[\TT_{\psi_0}g]$ to $[\GG]_{\KK}$ is the collection of $h^+(\O)$ lifts of the ATR cycles.\end{proposition}

\begin{proof} Let $h^+=h^+(\O[1/p])$ and let $t_1,...,t_{h^+}\in\A_K^\times$ be adelic representatives of $Pic^+(\O[1/p])$, meaning that we have \[\A_K^\times=\bigsqcup_{i=1}^{h^+}K^\times t_i(\O^S)^\times K_p^\times(K_\infty^\times)^2.\] The image of $\TT_{\psi_0}$ in $[\GG]_{\KK}$ is formed by the $\GG(F)$-orbits of the sets \[\psi_0(t_i^S)\KK^S\times\psi_0(t_{i,S}K_p^\times(K_\infty^\times)^2\KK_p.\] Under the identification \eqref{doublequot}, in $\PGL_2^+(\Z[1/p])\backslash \PGL_2^+(\Q_S)$ they correspond to the $\PGL_2^+(\Z[1/p])$-orbits of the sets $\gamma_i\psi_0(t_{i,S}K_p^\times(K_\infty^\times)^2)$, where $\gamma_i\psi_0(t_i)\in \KK^S\PGL_2^+(\Q_S)$. Applying the right translation by $g^{\psi_0}$, we obtain that the image of $[\TT_{\psi_0}g^{\psi_0}]$ in the space $\lquot{\PGL_2^+(\Z[1/p])}{\PGL^+_2(\Q_S)}$ is composed of the $\PGL^+_2(\Z[1/p])$-orbits of the sets \[(\gamma_i\psi_0\gamma_i^{-1})(K_p^\times (K_\infty^\times)^2)\gamma_i\psi_0(t_{i,S})g^{\psi_0}=(\gamma_i\psi_0(t_{i,\infty})g^{\psi_0}_\infty A^+(\R),\gamma_i\psi_0(t_{i,p})g^{\psi_0}_p\KK_p).\] By Proposition \ref{action}, the embeddings $\gamma_i\psi_0\gamma_i^{-1}$ form a complete set of representatives for $[opt(\O[1/p])]$ and since $\gamma_i\psi_0(t_{i\infty})\in\PGL_2^+(\R)$ we have $g^{\gamma_i\psi_0\gamma_i^{-1}}=\gamma_i\psi(t_{i,S})g^{\psi_0}$ (the conditions on the action over $\infty$ and $\alpha$ are satisfied immediately). This finishes the proof about the projection.
\end{proof}

\begin{proof}[Proof of Theorem \ref{equidistributionSH}]
We apply Theorem \ref{equidistribution} to the collection of tori $\{Y_{\O}=[\TT_{\psi_0}g]\}_\O$ with $R=M_2(\Z)$. Since the norm is surjective over the units group of unramified extensions, we have $\mathrm{nrd}(\KK_p)={\Z}_p^\times$ so that $S(\KK)=\emptyset$. There exists a unique integer $f$ coprime to $p$ such that $\O[1/p]=\O_f[1/p]$ with notation as in section \ref{orders}. Since $g^{\psi_0}_q=1$ for every prime $q\neq p$ we have that $\O_{Y_\O}=\O_{fp^n}$ for some $n\geq0$. From this we conclude that $({g^{\psi_0}_p}^{-1}\psi g_p^{\psi_0})(K_p)^\times\cap \PGL_2(\Z_p)$ is the projective image of the group of $n$-th principal units of $\mathcal{K}$. On the other hand, $({g_p^{\psi_0}}^{-1}\psi_0 g_p^{\psi_0})(K_p^\times)=\KK_p\subseteq\PGL_2(\Z_p)$ and $\mathcal{K}^\times/\Q_p^\times\cong \O_\mathcal{K}^\times/\Z_p^\times$, so $n=0$. Since $K$ is inert at $p$, the discriminant $d_K$ of $\O_K$ is coprime to $p$ and so $disc_p(\O)=d_Kf^2=disc(\O_Y)$. Using Proposition \ref{discriminant} we conclude that there is equidistribution as long as $disc_p(\O)$ goes to $\infty$.
\end{proof}

\selectlanguage{english}
\bibliographystyle{alpha}
\bibliography{ATRbiblio}

%\begin{thebibliography}{A}

%\bibitem[HMRL]{HMRL} S. Herrero, R. Menares \& J. Rivera-Letelier. \textit{$p$-adic distribution of CM points and Hecke orbits I: Convergence towards the Gauss point}. 2021.

%\bibitem[HMRL2]{HMRL2} S. Herrero, R. Menares \& J. Rivera-Letelier. \textit{$p$-adic distribution of CM points and Hecke orbits II: Linnik equidistribution on the supersingular locus}. 2021.

%\bibitem[ELMV2]{ELMV2} M. Einsiedler, E. Lindenstrauss, P. Michel \& A. Venkatesh. \textit{The distribution of closed geodesic on the modular surface, and Duke's theorem}. 2012.

%\bibitem[ALMW]{ALMW} M. Aka, M. Luethi, P. Michel \& A. Wieser. \textit{Simultaneous supersingular reductions of CM elliptic curves}. 2022. 

%\bibitem[BDIS]{BDIS} M. Bertolini, H. Darmon, A. Iovita \& M. Spiess. \textit{Teitelbaum's exceptional zero conjecture in the anticyclotomic setting}. 2002.

%\bibitem[BD]{BD} M. Bertolini \& H. Darmon. \textit{$p$-adic periods, $p$-adic $L$-functions and the $p$-adic uniformization of Shimura curves}. 1999.

%\bibitem[D]{D} H. Darmon. \textit{Integration on $\mathcal{H}_p\times \mathcal{H}$ and arithmetic applications}. 2001.

%\bibitem[DV]{DV} H. Darmon \& J. Vonk. \textit{Singular moduli for real quadratic fields: A rigid analytic approach}. 2021.

%\bibitem[K]{K} I. Khayutin. \textit{Joint Equidistribution of CM points}. 2019.

%\bibitem[EW]{EW} M. Einsiedler \& T. Ward.  \textit{Ergodic Theory with a view towards Number Theory}. 2010.

%\bibitem[N]{N} A. Nordentoft. \textit{Concentration of closed geodesics in the homology of modular curves}. 2022 (preprint).

%\end{thebibliography}

\end{document}